\documentclass[12pt,reqno]{amsart}
\usepackage{amssymb}
\usepackage{amsmath, mathtools}

\usepackage{amsthm}

\usepackage{amscd}

\newcommand{\RNum}[1]{\uppercase\expandafter{\romannumeral #1\relax}}

\usepackage{caption}

\usepackage[T2A]{fontenc}
\usepackage[utf8]{inputenc}
\usepackage[english]{babel}

\input{int.def} 

\usepackage[sort]{cite}
\usepackage{tikz-cd}
\usetikzlibrary{cd}
\usepackage{dirtytalk}
\usepackage[linktoc=page, colorlinks, linkcolor=blue, citecolor=blue]{hyperref}

\usepackage{xcolor}
\usepackage{centernot}

\usepackage{enumitem}

\usepackage{pgfplots}
\usepackage{multicol}

\pgfplotsset{compat=1.17}

\makeatletter
\renewenvironment{proof}[1][\proofname]{%
  \par\vspace{\topsep}%
  \normalfont\topsep6\p@\@plus6\p@\relax
  \trivlist
  \item[\hskip\labelsep\itshape #1\@addpunct{.}]\ignorespaces
}{%
  \endtrivlist
}
\makeatother

\numberwithin{equation}{section}

\DeclarePairedDelimiterX \ip[2]{\langle}{\rangle}{#1,#2}
\DeclarePairedDelimiterXPP \Prob[1]{\mathbb{P}}\{\}{}{ #1} 
\DeclarePairedDelimiterXPP \Probevent[1]{\mathbb{P}}(){}{#1} 

\DeclareMathOperator{\E}{\mathbb{E}}

\usepackage[mathcal]{euscript}

\usepackage{titlesec}
\titleformat{\section}[runin]{\bfseries}{\thesection.}{3pt}{}[.]

\DeclareMathOperator{\Cov}{Cov} \newcommand{\MTP}{\mathrm{MTP}_2}

\usepackage{geometry}
\newgeometry{vmargin={25mm}, hmargin={22mm,22mm}, footskip=10mm}   

\begin{document}
\title[MaxCut for \(\mathrm{MTP}_2\) Covariances
]{MaxCut for \(\mathrm{MTP}_2\) Covariances}

\author{Gleb Smirnov}
\address{
Mathematical Sciences Institute, 
Australian National University, Canberra, Australia}
\email{gleb.smirnov@anu.edu.au}


\begin{abstract}
Let $X=(X_1,\dots,X_n)\in\{0,1\}^n$ have a multivariate totally positive ($\mathrm{MTP}_2$) law. We prove that $$ \sum_{i<j}\mathbb{E}\left[\left|\mathrm{Cov}(X_i,X_j \mid X_{[n]\setminus\{i,j\}})\right|\right] \le n/2, $$ and more generally a weighted MaxCut inequality for the fully conditioned covariances. As an application, we confirm a conjecture of Allen and O'Donnell on correlation rounding for signed $\mathrm{MTP}_2$ laws.
\end{abstract}

\maketitle
\setcounter{section}{0}

\section{Main results}\label{intro}

Let \(\mu\) be a probability measure on \(\{0,1\}^n\), and let
\[
X=(X_1,\dots,X_n)\sim\mu.
\]
Recall that \(\mu\) is multivariate totally positive of order two (\(\MTP\)) if
\[
\mu\bigl(\min\{x,y\}\bigr)\,
\mu\bigl(\max\{x,y\}\bigr)
\ge
\mu(x)\mu(y)
\qquad
\forall x,y\in\{0,1\}^n,
\]
where \(\min\) and \(\max\) are taken coordinatewise.
\smallskip%

For \(A\subset[n]\), let
\[
X_A=(X_k)_{k\in A}.
\]
Define the symmetric matrix \(C=C(X)\in\mathbb R^{n\times n}\) by
\[
C_{ii}=0
\]
and, for \(i\ne j\),
\[
C_{ij}
=
\E\!\left[
\left|
\Cov\!\left(
X_i,X_j
\,\middle|\,
X_{[n]\setminus\{i,j\}}
\right)
\right|
\right].
\]
Thus \(C_{ij}\) is the average absolute covariance of \(X_i\) and \(X_j\)
after all other coordinates have been revealed.
\smallskip%

Our main result is the following weighted MaxCut inequality.

\begin{theorem}\label{thm:main}
Suppose that
\[
X=(X_1,\dots,X_n)\in\{0,1\}^n
\]
has an \(\MTP\) law. Then, for every collection of nonnegative weights
\(w_{ij}\),
\[
\sum_{i<j} w_{ij} C_{ij}
\le
2\operatorname{MaxCut}(G,w)-W,
\]
where
\[
W=\sum_{i<j}w_{ij}.
\]
Here \(G\) is the weighted graph on the vertex set \([n]\), with weight \(w_{ij}\) on the edge \(\{i,j\}\).
\end{theorem}
For a partition
\[
[n]=A\sqcup A^c,
\]
the weight of the corresponding cut is
\[
\operatorname{Cut}_w(A)
=
\sum_{\substack{i\in A\\ j\in A^c}} w_{ij},
\]
and \(\operatorname{MaxCut}(G,w)\) is the maximum weight of such a cut.
\smallskip%

The theorem also applies after arbitrary flips of the coordinates. Recall that the law of \(X\) is signed \(\MTP\) if there exist
\(\varepsilon_1,\ldots,\varepsilon_n\in\{0,1\}\) such that the random vector
\[
Y=(Y_1,\ldots,Y_n),
\qquad
Y_i=
\begin{cases}
X_i, & \varepsilon_i=0,\\
1-X_i, & \varepsilon_i=1,
\end{cases}
\]
has an \(\MTP\) law. Indeed, flipping a coordinate changes the sign of the corresponding
conditional covariance but not its absolute value. Hence:
\[
\left|
\Cov\!\left(
Y_i,Y_j
\,\middle|\,
Y_{[n]\setminus\{i,j\}}
\right)
\right|
=
\left|
\Cov\!\left(
X_i,X_j
\,\middle|\,
X_{[n]\setminus\{i,j\}}
\right)
\right|.
\]
Consequently, \(C(Y)=C(X)\).
\smallskip%

We now present an immediate unweighted consequence. Take
\[
w_{ij}=1.
\]
Then \(G\) is the complete graph with unit weights,
\[
W=\binom n2,
\]
and
\[
\operatorname{MaxCut}(G)
=
\left\lfloor\frac{n^2}{4}\right\rfloor.
\]
Theorem~\ref{thm:main} gives:
\[
\sum_{i<j}C_{ij}
\le
2\left\lfloor\frac{n^2}{4}\right\rfloor
-
\binom n2 \le
\frac n2.
\]
Consequently,
\begin{equation}\label{eq:average}
\sum_{i<j}
\E\!\left[
\left|
\Cov\!\left(
X_i,X_j
\,\middle|\,
X_{[n]\setminus\{i,j\}}
\right)
\right|
\right]
\le
\frac n2.
\end{equation}
The proof of Theorem~\ref{thm:main} is short and elementary, but only after passing to the language of quantum probability. Following Williams~\cite[Chapter~10]{Williams}, we associate to \(\mu\) its 
square-root \(n\)-qubit state. The \(\MTP\) condition becomes log-supermodularity of the amplitudes. We first prove the weighted MaxCut inequality for suitable two-qubit
correlations of this state. We then bound the conditional covariances by these quantum correlations and transfer the MaxCut inequality back to \(C\).
\smallskip%

Despite its simple form, Theorem~\ref{thm:main} has applications in theoretical computer science. In \S\,\ref{ODonnell}, we use
\eqref{eq:average} to confirm a conjecture of Allen, O'Donnell, and Zhou for signed \(\MTP\) laws.
\smallskip%

We shall use the standard fact that \(\MTP\) is preserved under marginalization. See~\cite{Karlin-Rinott}. On the Boolean lattice, the \(\MTP\) condition is the classical FKG lattice condition. See~\cite{FKG, Grimmett}.

\medskip
\noindent\textbf{Acknowledgment and AI assistance.}
The author thanks OpenAI's ChatGPT for drawing his attention to Conjectures A and B in \cite{AllenODonnell}, and for pointing out that the results of this paper advance Conjecture A and fully prove Conjecture B. The proof of Conjecture B is not included here, but it is not difficult.

\section{Commuting Pauli observables}\label{quantum}
We follow Williams~\cite[Chapter~10]{Williams} 
for the elementary 
quantum-probability formalism used below.
\smallskip%

Let 
\[ 
H_n=(\mathbb C^2)^{\otimes n}, \qquad e_0= \begin{pmatrix}1\\0\end{pmatrix}, \qquad e_1= \begin{pmatrix}0\\1\end{pmatrix}. 
\] 
For \(x=(x_1,\ldots,x_n)\in\{0,1\}^n\), write: 
\[ 
e_x=e_{x_1}\otimes\cdots\otimes e_{x_n}. 
\] 
The vectors \(e_x\) 
form the computational basis of \(H_n\).
\smallskip%

A pure state is a unit vector 
\[ 
\psi=\sum_x\psi(x)e_x. 
\] 
Measurement in the computational basis produces \(x\) with probability 
\[ 
|\psi(x)|^2. 
\] 
If \(A\) is a Hermitian operator on \(H_n\), 
its mean in the state \(\psi\) is 
\[ 
\E_\psi[A]=(\psi,A\psi). 
\]
We use the Pauli matrix: 
\[ 
Y= \begin{pmatrix} 0&-i\\ i&0 \end{pmatrix}.
\] 
It satisfies:
\[ 
Y^*=Y, \qquad Y^2=I, \qquad Ye_0=ie_1, \qquad Ye_1=-ie_0. 
\]
Let \(Y_i\) denote \(Y\) acting on the \(i\)-th tensor factor and the identity on the others. 
Then: 
\[ 
Y_i^*=Y_i, \qquad Y_i^2=I, \qquad Y_iY_j=Y_jY_i. 
\]
In particular, \(Y_i\) are commuting Hermitian operators, and each has eigenvalues \(\pm1\).
\smallskip%

A state 
\[ 
\psi=\sum_x\psi(x)e_x 
\] 
is called positive log-supermodular if \(\psi(x)\ge0\) for every \(x\) and 
\[ 
\psi(x\wedge y)\psi(x\vee y) \ge \psi(x)\psi(y) \qquad (x,y\in\{0,1\}^n). 
\]
Recall the quantum analogue of the law of total
expectation.
\begin{lemma}\label{total}
Let \(\psi\) be a unit state. 
Let \((P_z)_{z \in \mathcal Z}\) be
orthogonal projections such that:
\[
\sum_zP_z=I.
\]
Set:
\[
p_z=(\psi,P_z\psi).
\]
For \(p_z>0\), let
\[
\psi_z=\frac{P_z\psi}{\sqrt{p_z}}
\]
be the state conditioned on the measurement outcome \(z\).
If \(A\) commutes with every \(P_z\), then:
\[
(\psi,A\psi)
=
\sum_{z}
p_z(\psi_z,A\psi_z),
\]
where terms \(p_z=0\) are omitted.
\end{lemma}
\begin{proof}
We calculate:
\begin{align*}
\sum_z p_z(\psi_z,A\psi_z)
&=
\sum_z
\bigl(\sqrt{p_z}\,\psi_z,
      \sqrt{p_z}\,A\psi_z\bigr)
&& 
\\
&=
\sum_z
(P_z\psi,AP_z\psi)
&& \text{\(\sqrt{p_z}\,\psi_z=P_z\psi\)}
\\
&=
\sum_z
(\psi,P_zAP_z\psi)
&& \text{\(P_z^*=P_z\)}
\\
&=
\sum_z
(\psi,P_z^2A\psi)
&& \text{\(AP_z=P_zA\)}
\\
&=
\sum_z
(\psi,P_zA\psi)
&& \text{\(P_z^2=P_z\)}
\\
&=
(\psi,A\psi)
&& \text{\(\sum_zP_z=I\),}
\end{align*}
and the proof follows. \qed
\end{proof}
We now use the commutativity of \(Y_1,\ldots,Y_n\).
\begin{lemma}\label{joint}
Let \(\psi\) be a unit state. Measure the commuting observables \(Y_1,\ldots,Y_n\) simultaneously in the state \(\psi\). Since each \(Y_i\) has eigenvalues \(\pm1\), 
the outcome is a random vector
\[
S=(S_1,\ldots,S_n)\in\{-1,1\}^n.
\]
Then:
\[
\E[S_iS_j]
=
(\psi,Y_iY_j\psi).
\]
\end{lemma}
\begin{proof}
Since \(Y_i\) are commuting and Hermitian,
they admit a common orthonormal eigenbasis. 
In this basis, the eigenvalue of \(Y_iY_j\)
is the product \(S_iS_j\); thus, the formula for 
the mean of an observable gives:
\[
\E[S_iS_j]
=
\E_\psi[Y_iY_j]
=
(\psi,Y_iY_j\psi).
\]
\qed
\end{proof}
We prove the quantum result underlying
Theorem~\ref{thm:main}.
\begin{theorem}\label{thm:quantum}
Let
\[
\psi
=
\sum_{x\in\{0,1\}^n}
\psi(x)e_x
\]
be a positive log-supermodular state. Define \(L=L(\psi)\in\mathbb R^{n\times n}\) by
\[
L_{ii}=0,
\qquad
L_{ij}=-(\psi,Y_iY_j\psi)
\quad \text{for \(i \neq j\)}.
\]
Then:
\[
L_{ij} \ge 0
\]
and for each collection of nonnegative weights \(w_{ij}\),
\[
\sum_{i<j}w_{ij}L_{ij}
\le
2\,\operatorname{MaxCut}(G,w)-W,
\]
where
\[
W=\sum_{i<j}w_{ij}.
\]
\end{theorem}

\begin{proof}
Fix distinct \(i,j\in[n]\). 
Measure all qubits other than \(i,j\)
in the computational basis.
An outcome is a configuration
\[
z\in\{0,1\}^{[n]\setminus\{i,j\}}.
\]
\(z\) specifies the values of all coordinates
other than \(i\) and \(j\). Let
\[
H_z
=
\operatorname{span}\{e_{abz}:a,b\in\{0,1\}\},
\]
where \(e_{abz}\) is the computational-basis vector
with \(i\)-th coordinate \(a\),
\(j\)-th coordinate \(b\),
and remaining coordinates \(z\).
Let \(P_z\) be the orthogonal projection onto \(H_z\).
Set:
\[
p_z=(\psi,P_z\psi).
\]
For \(p_z>0\), the conditioned state is
\begin{equation}\label{psi_z}
\psi_z
=
\frac{P_z\psi}{\sqrt{p_z}}
=
\alpha_z e_{11z}
+\beta_z e_{10z}
+\gamma_z e_{01z}
+\delta_z e_{00z}.
\end{equation}
Log-supermodularity on this two-dimensional face gives:
\[
\alpha_z\delta_z
\ge
\beta_z\gamma_z.
\]
\smallskip%

On \(H_z\), the operator \(Y_iY_j\) acts by
\[
Y_iY_j e_{11z}=-e_{00z},
\quad
Y_iY_j e_{10z}=e_{01z},
\quad
Y_iY_j e_{01z}=e_{10z},
\quad
Y_iY_j e_{00z}=-e_{11z}.
\]
Hence:
\[
(\psi_z,Y_iY_j\psi_z)
=
2(\beta_z\gamma_z-\alpha_z\delta_z)
\le0.
\]
Since \(Y_iY_j\) acts only on qubits \(i,j\),
it commutes with every \(P_z\). Lemma~\ref{total} gives:
\[
-L_{ij}
=
(\psi,Y_iY_j\psi)
=
\sum_z
p_z(\psi_z,Y_iY_j\psi_z)
\le0.
\]
Let again
\[
S=(S_1,\ldots,S_n)\in\{-1,1\}^n
\]
be the simultaneous measurement of 
\(Y_1, \ldots, Y_n\). By Lemma~\ref{joint},
\[
\E[S_iS_j]
=
(\psi,Y_iY_j\psi)
=
-L_{ij},
\]
while
\[
\E[S_i^2]=1.
\]
Let \(w_{ij}\ge0\). For \(s\in\{-1,1\}^n\), write:
\[
\operatorname{Cut}_w(s)
=
\sum_{i<j}
w_{ij}\frac{1-s_is_j}{2}.
\]
Then:
\[
-\sum_{i<j}w_{ij}s_is_j
=
2\operatorname{Cut}_w(s)-W.
\]
Consequently,
\[
\sum_{i<j}w_{ij}L_{ij} = 
\E\!\left[
2\operatorname{Cut}_w(S)-W
\right] \le 2\operatorname{MaxCut}(G,w)-W,
\]
and the proof follows.
\end{proof}
\begin{remark}
Log-supermodularity is used only to 
prove \((\psi,Y_iY_j\psi)\le0\). It is not needed for the MaxCut bound.
\end{remark}

\section{From quantum to classical covariances} \label{classical}
Let \(X\sim\mu\), where \(\mu\) 
is \(\MTP\). Consider the square-root state:
\[
\psi
=
\sum_{x\in\{0,1\}^n}
\sqrt{\mu(x)}\,e_x.
\]
Since
\[
\mu(x\wedge y)\mu(x\vee y)
\ge
\mu(x)\mu(y)
\]
is equivalent to
\[
\sqrt{\mu(x\wedge y)}
\sqrt{\mu(x\vee y)}
\ge
\sqrt{\mu(x)}
\sqrt{\mu(y)},
\]
the state \(\psi\) is positive 
log-supermodular.
\smallskip%

Let 
\[ 
L_{ii}=0, \qquad L_{ij}=-(\psi,Y_iY_j\psi) \quad(i\ne j). 
\] 
We compare \(L_{ij}\) with the conditional covariance of
\(X_i\) and \(X_j\).
\begin{lemma}\label{2qubit}
Let \(U,V\in\{0,1\}\) have joint distribution
\[
\begin{array}{c|cc}
 & V=1 & V=0\\ \hline
U=1 & a & b\\
U=0 & c & d
\end{array},
\qquad
a+b+c+d=1,
\]
and suppose
\[
ad\ge bc.
\]
Let
\[
\phi
=
\sqrt a\,e_{11}
+\sqrt b\,e_{10}
+\sqrt c\,e_{01}
+\sqrt d\,e_{00}.
\]
Then:
\[
0\le\Cov(U,V)
\le
-(\phi,(Y\otimes Y)\phi).
\]
\end{lemma}
\begin{proof}
Since
\[
\Cov(U,V)=ad-bc,
\]
we have
\[
\Cov(U,V)
=
(\sqrt{ad}-\sqrt{bc})(\sqrt{ad}+\sqrt{bc})
\le
2(\sqrt{ad}-\sqrt{bc}).
\]
On the other hand,
\[
-(\phi,(Y\otimes Y)\phi)
=
2(\sqrt{ad}-\sqrt{bc}),
\]
and the proof follows.
\end{proof}

\begin{proof}[Proof of Theorem~\ref{thm:main}]
Fix distinct \(i,j\). Set:
\[
Z=X_{[n]\setminus\{i,j\}}.
\]
For an outcome \(Z=z\), write the conditional law of \((X_i,X_j)\) as:
\[
\begin{array}{c|cc}
 & X_j=1 & X_j=0\\ \hline
X_i=1 & a_z & b_z\\
X_i=0 & c_z & d_z.
\end{array}
\]
The \(\MTP\) inequality on this face gives: 
\[ 
a_zd_z\ge b_zc_z. 
\]
Let again
\[
p_z=\mathbb P(Z=z),
\]
and let \(\psi_z\) be the conditioned state as in \eqref{psi_z}:
\[
\psi_z = \alpha_z e_{11z} +\beta_z e_{10z} +\gamma_z e_{01z} +\delta_z e_{00z},\quad 
\alpha_z = \sqrt{a_z},\  
\beta_z = \sqrt{b_z},\ 
\gamma_z = \sqrt{c_z},\ 
\delta_z = \sqrt{d_z}.
\]
From Lemma~\ref{2qubit}:
\[
0 \le 
\Cov(X_i,X_j\mid Z=z) \le 
-(\psi_z,(Y\otimes Y)\psi_z),
\]
Then:
\[
0 \le C_{ij} = \E\!\left[
\Cov(X_i,X_j\mid Z)
\right] \le 
-\sum_zp_z
(\psi_z,(Y\otimes Y)\psi_z).
\]
From Lemma~\ref{total}:
\[
C_{ij} \le -(\psi,Y_iY_j\psi) = L_{ij}.
\]
For every collection of nonnegative weights 
\(w_{ij}\), 
\[ 
\sum_{i<j}w_{ij}C_{ij} \le \sum_{i<j}w_{ij}L_{ij}. 
\] 
Using Theorem~\ref{thm:quantum}, we obtain: 
\[ 
\sum_{i<j}w_{ij}C_{ij} \le 
2\,\operatorname{MaxCut}(G,w)-W,\quad 
W=\sum_{i<j}w_{ij},
\] 
and the proof follows. \qed
\end{proof}

\section{On a conjecture of Allen, 
O'Donnell, and Zhou}\label{ODonnell}

Allen, O'Donnell, and Zhou \cite{AllenODonnell} introduced:
\[
\operatorname{avgCov}_{\mid t}(X)
:=
\operatorname*{avg}_{\substack{J\subset[n]\\ |J|=t}}
\operatorname*{avg}_{\substack{i,j\in[n]\setminus J\\ i\ne j}}
\E\!\left[
\left|
\Cov\!\left(
X_i,X_j
\,\middle|\,
X_J
\right)
\right|
\right].
\]
They conjectured that for every
\(\varepsilon>0\), there exists
\[
t=O(1/\varepsilon)
\]
such that:
\begin{equation}\label{eq:conj_A}
\operatorname{avgCov}_{\mid t}(X)\le\varepsilon.
\end{equation}
This is their Conjecture A. 
\smallskip%

Raghavendra and Tan~\cite{RagTan} proved the weaker estimate:
\[
t=O(\varepsilon^{-2}).
\]
The conjecture is false for general measures, as 
shown by Jain, Koehler, and Risteski~\cite{Jain19}. For \(\MTP\) measures, however, it follows from 
\eqref{eq:average}.
\begin{proof}
Setting
\[
U = J \cup \{i,j\},
\]
we have
\[
\operatorname{avgCov}_{\mid t}(X)
=
\operatorname*{avg}_{\substack{U\subset[n]\\ |U|=t+2}}
\operatorname*{avg}_{\substack{i,j\in U\\ i\ne j}}
\E\!\left[
\left|
\Cov\!\left(
X_i,X_j
\,\middle|\,
X_{U\setminus\{i,j\}}
\right)
\right|
\right].
\]
Since \(\MTP\) is preserved under marginalization, \(X_U\) has an \(\MTP\) law. From
\eqref{eq:average}:
\[
\sum_{\substack{i<j\\ i,j\in U}}
\E\!\left[
\left|
\Cov\!\left(
X_i,X_j
\,\middle|\,
X_{U\setminus\{i,j\}}
\right)
\right|
\right]
\le
\frac{t+2}{2}.
\]
Since there are
\[
\frac{(t+1)(t+2)}{2}
\]
pairs in \(U\), it follows that:
\[
\operatorname*{avg}_{\substack{i,j\in U\\ i\ne j}}
\E\!\left[
\left|
\Cov\!\left(
X_i,X_j
\,\middle|\,
X_{U\setminus\{i,j\}}
\right)
\right|
\right]
\le
\frac{1}{t+1}.
\]
Averaging over \(U\):
\[
\operatorname{avgCov}_{\mid t}(X)
\le
\frac{1}{t+1},
\]
and \eqref{eq:conj_A} follows. \qed
\end{proof}

\bibliographystyle{plain}
\bibliography{ref}

\end{document}